\documentclass[submission,copyright,creativecommons]{eptcs}
\providecommand{\event}{Applied Category Theory 2022}
\usepackage{underscore}           
\usepackage{tikz-cd}
\usepackage{amsmath}
\usepackage{amsthm}
\usepackage{amssymb}
\usepackage{stmaryrd}

\newtheorem{theorem}{Theorem}[section]

\newtheorem{lemma}[theorem]{Lemma}

\newtheorem{proposition}[theorem]{Proposition}

\theoremstyle{definition}

\newtheorem{examplex}[theorem]{Example}
\newtheorem{definition}[theorem]{Definition}

\newenvironment{example}
  {\pushQED{\qed}\examplex}{\popQED\endexamplex}

\newcommand{\mc}[1]{\mathcal{#1}}

\newcommand{\mb}[1]{\mathbf{#1}}

\newcommand{\ms}[1]{\mathsf{#1}}

\newcommand{\mr}[1]{\mathrm{#1}}

\newcommand{\set}[1]{\{#1\}}
\newcommand{\geo}[1]{\left|#1\right|}

\newcommand{\scomp}[2]{ \{\, #1 \mid #2 \,\}}

\newcommand{\op}{^{\mr{op}}}

\newcommand{\id}{\operatorname{id}}
\newcommand{\pair}[1]{\langle #1 \rangle}

\newcommand{\nt}{\Rightarrow}

\newcommand{\eff}{\Leftrightarrow}

\newcommand{\ass}[1]{\llbracket#1\rrbracket} 

\newcommand{\dg}{^\dagger}

\newcommand{\inv}{^{-1}}

\newcommand{\hook}{\hookrightarrow}

\renewcommand{\lim}{\operatorname{Lim}}

\newcommand{\Set}{\mathbf{Set}}

\newcommand{\Top}{\mathbf{Top}}

\newcommand{\Cat}{\mathbf{Cat}}

\newcommand{\Pre}{\mathbf{Pre}}

\newcommand{\SupL}{\mathbf{SupL}}
\newcommand{\InfL}{\mathbf{InfL}}

\newcommand{\CABAO}{\mb{CABAO}}

\newcommand{\upa}{^+}

\newcommand{\Kr}{\mb{Kr}}

\newcommand{\Nb}{\mb{Nb}}

\newcommand{\Evl}{\mb{Evl}}
\newcommand{\Eqv}{\mb{Eqv}}

\newcommand{\mta}{^{\Sigma}}
\newcommand{\assv}[1]{\ass{#1}^V}
\newcommand{\PAL}{^{\mb{PAL}}}

\newcommand{\PRO}{^\mb{PRO}}

\title{Unification of Modal Logic via Topological Categories}
\author{Lingyuan Ye
\institute{Tsinghua University \\ Beijing, China}
\email{ye.lingyuan.ac@gmail.com} }
\def\titlerunning{Unification of Modal Logic via Topological
 Categories}
\def\authorrunning{L. Ye}

\begin{document}
\maketitle

\begin{abstract}
  In this paper we provide a unifying description of different types
  of semantics of modal logic found in the literature via the
  framework of topological categories. In the style of categorical
  logic, we establish an exact correspondence between various
  syntactic extensions of modal logic on one hand, including modal
  dependence, group agent structures, and logical dynamic, and
  semantic structures in topological categories on the other
  hand. This framework provides us a uniform treatment of interpreting
  these syntactic extensions in all different types of semantics of
  modal logic, and it deepens our conceptual understanding of the
  abstract structure of modal logic.
\end{abstract}

\section{Introduction}
\label{sec:intro}

Throughout the history of modal logic, many different types of
semantics have been developed to interpret the modal language, with
various applications in mind. Starting from the seminal work by von
Wright~\cite{vonwright1954essay} and the later extension by Hintikka
in~\cite{hintikka1962knowledge}, the Kripkean style semantics of modal
logic has been widely applied in the philosophical study of
epistemology.  Tarski and McKinsey in~\cite{mckinsey1944algebra} have
also discovered that the interior operator induced by a topological
space could be used to interpret modal formulas as well, which
naturally finds its connection with propositional intuitionistic
logic. Other variations include neighbourhood semantics for modal
logic, first suggested by Scott in~\cite{scott1970advice} in order to
study certain non-normal fragments of modal logic. Finally, we also
have semantics of a more algebraic flavour, extending the usual
algebraisation of propositional logic using Boolean algebras.

These various forms then naturally bear the following question: Is it
possible to provide a \emph{unifying} description of all types of
semantic models of modal logic? To provide a positive answer, this
paper starts with the following observation: In all of the above
mentioned examples, in fact in many more cases, the categories of
semantics of modal logic all organise themselves into
\emph{topological categories} (over $\Set$).

The notion of a topological category is introduced
in~\cite{adamek1990abstract}, with the aim of axiomatising the
structure of those categories containing objects $X$ equipped with
certain \emph{geometric data}, with $X$ living in an ambient category
$\mb X$. This results in the notion of topological categories over an
arbitrary base $\mb X$. For our purpose though, we will exclusively
work over $\Set$, and this is our default for topological categories
henceforth. The prototypical example is $\Top$, the category of
topological spaces, whose objects are sets equipped with a topology.
We will give an overview of topological categories in
Section~\ref{sec:preliminary}, and provide another equivalent way of
describing topological categories more suitable for modal logic
(cf.~Theorem\ref{thm:topocatindex}). According to this theorem, it can
then be immediately recognised that all the mentioned examples of
semantics conform to such a description: Kripke models are sets
equipped with a binary relation, which are often depicted
diagrammatically. We've already mentioned topological spaces, and
neighbourhood models are no exceptions. Perhaps surprisingly, a
particular style of algebraic semantics, using \emph{complete atomic
  Boolean algebra with operators} (CABAO), can also be recognised as
topological or geometrical over $\Set$, once we take its dual
category. This is arguably an incarnation of the duality principle
between algebra and geometry within the context of modal logic. We
will prove in Proposition~\ref{prop:alltopocat} that all these types
of semantics, and in fact much more, are instances of topological
categories, hence building the foundations of unification.

But such fact alone is far from convincing that this is a good
framework for unifying modal logic. The more important topic is how
the semantic structures of topological categories would explain the
various logical features that are present in a modal context. In this
paper, we will follow the philosophy of categorical logic,
establishing \emph{exact correspondences} between different syntactic
patterns of modal logic with semantic structures of topological
categories. Such correspondences are witnessed by considering
transformation of models, viz. functors between topological
categories.

The first thing to explain is the interpretation of modalities. As we
will see in more detail in Section~\ref{sec:modality}, it is precisely
the geometric data of a topological category that is responsible for
its interpretation. Furthermore, the structure of topological
categories also connects tightly with many other \emph{extensions} of
basic modal logic studied in the literature, including the
multi-agency, group agency, modal dependence, logical dynamics,
etc.. For each of these reasoning patterns we have established
theorems (see
Theorem~\ref{thm:modalfuncpreserveL},~\ref{thm:conservextenDSigma}
and~\ref{thm:preserveproducttypeupdate}), showing that functors
preserve certain structures of topological categories if, and only if,
the linguistic interpretation of the corresponding fragment of modal
logic remains unchanged under the transformation. These results
significantly improve our conceptual understanding of modal logic, and
will be the main topics of Section~\ref{sec:group}
and~\ref{sec:dynamics}.

To the best knowledge of the author, in the current literature there
has been no theoretic framework to enable all these different
fragments of modal logic to be described in a uniform way for all
types of semantics. Our systematic approach allows seamless
generalisation of all these constructions in modal logic to any other
semantics. For instance, it has been actively discussed what is the
corresponding notion of common knowledge in topological
semantics~\cite{topoapproach2019}, how to extend different forms of
logical dynamics to wider contexts~\cite{van2011dynamic}, or how to
develop modal dependence described in~\cite{baltag2021simple}
and~\cite{baltag2021lcd} for other semantic types. Our work provides a
novel answer to all these different questions by accommodating them to
the framework of topological categories, and it has ample potential
applications. 

\section{Preliminaries}
\label{sec:preliminary}
In many existing texts,
e.g. in~\cite{adamek1990abstract,hofmann2014monoidal}, topological
categories are usually introduced as \emph{fibrations} over $\Set$
satisfying certain lifting properties. It is well-known from the
Grothendieck construction that fibrations can be equivalently
described by \emph{indexing categories}, or functors mapping out of
$\Set$. For our purpose, it is this equivalent \emph{indexing} point
of view of topological categories that is more suitable for making
connections with modal logic. We will discuss this in more detail
below.

Recall that a \emph{concrete category}, or a \emph{construct}, is
simply a faithful functor $U : \mc A \to \Set$. When it is clear from
the context what the functor $U$ is, we will simply refer to $\mc A$
as a concrete category.

\begin{example}
\label{exam:examcat} We take this opportunity to introduce the main
examples of category of semantics:
  \begin{itemize}
  \item $\Kr$ denotes the category of Kripke frames, whose objects are
    sets equipped with a binary relation on them, with morphisms being
    monotone maps. It has certain useful full subcategories including
    $\Pre$ and $\Eqv$, whose objects only contains preorders or
    equivalence relations.

  \item We've mentioned that $\Top$ will denote the category of
    topological spaces.

  \item $\Nb$ is the category of neighbourhood frames, whose objects
    are sets $X$ equipped with a neighbourhood relation
    $E \subseteq X \times \wp(X)$, and whose morphisms
    $f : (X,E) \to (Y,F)$ are functions from $X$ to $Y$ satisfying a
    continuity condition: For any $x\in X$ and $V \subseteq Y$,
    $fx F V \nt x E f\inv V$.

  \item We let objects of $\CABAO$ be pairs $(X,m)$ with $m$ being an
    arbitrary endo-function on $\wp(X)$, and morphisms
    $f : (X,m) \to (Y,n)$ are functions from $X$ to $Y$ satisfying
    $f\inv \circ n \subseteq m \circ f\inv$, where we extend the order
    $\subseteq$ on $\wp(X)$ point-wise to the function space
    $\wp(X)^{\wp(Y)}$.\footnote{This definition of the category
      $\CABAO$ contains certain subtle points, which we will explain
      in a minute.}

  \item Besides models of the above form, to interpret modal formulas
    we also need evaluation functions to interpret propositional
    letters. For a fixed set $\ms P$ of propositional variables, we
    introduce the category $\Evl$ of evaluations, whose objects are
    pairs $(X,V)$ with $V : \ms P \to \wp(X)$, and morphisms
    $f : (X,V) \to (Y,W)$ are functions from $X$ to $Y$ satisfying
    $V \subseteq f\inv \circ W$, where similarly the order is the
    point-wise extension of the subset relation on the function space
    $\wp(X)^\ms P$.
  \end{itemize}
  In each case, there is an evident forgetful functor to $\Set$ that
  identifies them as concrete categories.
\end{example}

Let us say a few more words on the category $\CABAO$. From a
well-known theorem of Tarski, we know every CABA is isomorphic to a
power set algebra $\wp(X)$ (and every power set algebra is a CABA),
and every morphism between them is of the form
$f\inv : \wp(Y) \to \wp(X)$ for some function $f : X \to Y$. Hence,
our definition of a CABAO as a pair $(X,m)$ does not lose anything,
and it builds in the duality, since it uses $f$, rather than $f\inv$,
as morphisms. Notice that the morphisms we choose between CABAOs are
not the \emph{algebraic} ones, which should commute with the operators
on both sides, but \emph{lax} ones that only require an inequality. A
possible intuition for this choice is to read the operators $m,n$ as
interior operators induced by a topology, and the above continuity
condition is exactly saying that $f$ is a continuous map for the two
topological spaces. We will see later that such a choice makes
$\CABAO$ topological over $\Set$.

There is also an accompanying notion of \emph{concrete functors}
between concrete categories: A functor $F$ between two concrete
categories $(\mc A,\geo-_\mc A)$ and $(\mc B,\geo-_\mc B)$ is a
concrete functor iff it commutes with the forgetful functors, i.e. iff
it preserves the underlying sets. Obviously, each forgetful functor of
$\geo-$ of a construct $\mc A$ constitute a concrete functor from
$(\mc A,\geo-)$ to $(\Set,1_\Set)$, which establish $\Set$ as the
terminal object in the (large) category of concrete categories and
concrete functors.

The faithfulness of the forgetful functor of a concrete category has
many consequences. For any construct $(\mc A,\geo-)$, we will identify
the Hom-sets $\mc A(A,B)$ simply as subsets of $\Set(\geo A,\geo B)$,
and say a function $f : \geo A \to \geo B$ is an
\emph{$\mc A$-morphism} if it belongs to $\mc A(A,B)$. For instance,
$f$ is a $\Top$-morphism if it is continuous. Faithfulness of $\geo-$
also implies that each fibre $\mc A_X$ over a set $X$ is a (possibly
large) preorder --- recall that a morphism in $\mc A_X$ is a morphism
in $\mc A$ above $\id_X$. If each fibre is indeed small, then we say
the construct $\mc A$ is \emph{fibre-small}. It is easy to verify that
all the introduced categories in Example~\ref{exam:examcat} have small
fibres. All the constructs considered in the future will be
fibre-small.

As mentioned, topological categories are constructs that satisfy
certain lifting properties. For any construct $(\mc A,\geo-)$, a
\emph{structured source} is defined to be a set of functions of the
form $\set{f_i : X \to \geo{A_i}}_{i\in I}$,\footnote{If we don't
restrict to fibre-small constructs, then we need to consider
structured sources whose size are \emph{proper classes}. However, this
is not a problem for us to worry about. We refer the readers
to~\cite{adamek1990abstract} for more details.} where each $A_i\in\mc
A$. An \emph{initial lift} of such a structured source is an object
$A$ in the fibre $\mc A_X$, satisfying the following universal
properties: For any function $g : \geo B \to \geo A$, $g$ is an $\mc
A$-morphism iff $f_i \circ g : \geo B \to \geo{A_i}$ is an $\mc
A$-morphism for any $i\in I$. Evidently, initial lifts are identified
up to isomorphisms in the fibre $\mc A_X$.

\begin{definition}[Topological Categories]
  A construct $(\mc A,\geo-)$ is a \emph{topological category} if
  every structured source has a \emph{unique} initial lift.
\end{definition}

We can break the definition of a topological category into two parts:
It first requires the \emph{existence} of initial lifts of structured
sources, and it also requires the \emph{uniqueness} of such lifts. The
notion of initial lift of structured source is a generalisation of
cartesian lifts for Grothendieck fibrations. In fact, cartesian lift
is exactly initial lift for a \emph{singleton structured source},
viz. a structured source consisting of only one function. This in
particular suggests that topological categories are special types of
\emph{fibrations} where we can perform lifts against an arbitrary set
of morphisms with a common codomain. Together with the uniqueness part
of the definition, a topological category satisfies many desirable
properties:\footnote{The following two lemmas are both contained in
  \cite{adamek1990abstract}. We include the proof here for the
  convenience of the readers.}

\begin{lemma}
  If $\mc A$ is a topological category, then each fibre $\mc A_X$ is a
  complete lattice for any set $X$.
\end{lemma}
\begin{proof}
  For any family $\set{A_i}_{i\in I}$ in the fibre $\mc A_X$, consider
  the structured source, $\set{1_X : X \to \geo{A_i}}_{i\in I}$. It is
  routine to verify that its unique initial lift is precisely the meet
  of this family in $\mc A_X$.
\end{proof}

The existence of initial lifts guarantees each fibre to be complete
preorders, and the uniqueness then implies that they are indeed
posets. As a fibration, given any function $f : X \to Y$, the initial
lifts along $f$ will induce functions of the form $f^* : \mc A_Y \to
\mc A_X$. Again, $f^*$ being a well-defined function is guaranteed by
the uniqueness of initial lifts, and we will also denote maps of the
form $f^*$ as \emph{pullback maps}. Furthermore, uniqueness also
suggests that the fibration \emph{splits}, in the sense that $1_X^* =
1_{\mc A_X}$ and $g^*f^* = (gf)^*$. The more important observation is
that each pullback map preserves meets in the fibre:

\begin{lemma}
  Let $(\mc A,\geo-)$ be a topological category, then for any function
  $f : X \to Y$, the pullback map $f^* : \mc A_Y \to \mc A_X$
  preserves arbitrary meets.
\end{lemma}
\begin{proof}
  For any family $\set{B_i}_{i\in I}$ in $\mc A_Y$, we only need to
  prove $\bigwedge_{i\in I}f^{*}B_i \le f^{*}\bigwedge_{i\in
    I}B_i$. By definition, this holds iff the identity function,
  viewed as a map
  $1_X : \geo{\bigwedge_{i\in I}f^*B_i} \to \geo{f^*\bigwedge_{i\in
      I}B_i}$, is an $\mc A$-morphism. By the universal property of
  initial lift, it is so iff
  $f \circ 1_X = f : \geo{\bigwedge_{i\in I}f^*B_i} \to
  \geo{\bigwedge_{i\in I}B_i}$ is an $\mc A$-morphism, and again, this
  is furthermore equivalent to all the maps in the structured source
  $\left\{f : \geo{\bigwedge_{i\in I}f^*B_i} \to \geo{B_i}\right\}$
  being $\mc A$-morphisms. However, we know that
  $\bigwedge_{i\in I}f^*B_i \le f^*B_i$ for any $i\in I$, which means
  both $1_X : \geo{\bigwedge_{i\in I}f^*B_i} \to \geo{f^*B_i}$ and
  $f : \geo{f^*B_i} \to \geo{B_i}$ are $\mc A$-morphisms, hence so is
  the composite. 
\end{proof}

It follows that each pullback map $f^*$ has a unique left adjoint,
which we denote as $f_!$ and call it the \emph{pushforward map}. By
the adjunction $f_! \dashv f^*$ and the universal property of initial
lift, it is easy to see that $f_!$ are exactly describing the
cocartesian lifts, which makes a topological category an
\emph{opfibration} as well, hence a \emph{bifibration}. As
Theorem~\ref{thm:topocatindex} will show, the data of fibres and
pullback or pushforward maps uniquely determines a topological
category:
\begin{theorem}
  \label{thm:topocatindex}
  Let $\InfL$ (resp. $\SupL$) be the category of inflattices
  (suplattices).\footnote{$\InfL$ (resp. $\SupL$) is the category of
    complete lattices with meet (resp. join) preserving maps. For more
    detailed description of various categorical structures on $\InfL$
    or $\SupL$, we refer the readers to~\cite[Chapter
    I]{joyal1984extension}.} Recall that they are canonically dual to
  each other. The data of a topological category $(\mc A,\geo-)$ is
  the same as the data of a functor $\mc A_{(-)} : \Set\op \to \InfL$,
  or equivalently $\mc A_{-} : \Set \to \SupL$.
\end{theorem}
\begin{proof}
  We've already shown that a topological category induces a functor
  from $\Set\op$ to $\InfL$. On the other hand, since $\InfL$ is a
  subcategory of $\Cat$, any functor $F : \Set\op \to \InfL$ admits a
  Grothendieck construction, resulting in a fibration
  $p : \mc F \to \Set$. The objects of $\mc F$ are pairs $(X,A)$ with
  $A$ being an element in $F(X)$; a morphism $f : (X,A) \to (Y,B)$ is
  a function $f : X \to Y$, such that $A \le Ff(B)$. The forgetful
  functor $p$ is evident. To this end, we only need to verify that for
  arbitrary structured source $\set{f_i : X \to p(X_i,A_i)}_{i\in I}$,
  it has a unique initial lift, which we claim is given by
  $\bigwedge_{i\in I}(Ff_i)(A_i)$ over $X$. For any function
  $g : p(Y,B) \to p(X,\bigwedge_{i\in I}(Ff_i)(A_i))$, by definition
  it is an $\mc F$-morphism iff
  \[ B \le Fg\bigwedge_{i\in I}(Ff_i)(A_i) = \bigwedge_{i\in
      I}F(f_i\circ g)(A_i) \eff \forall i\in I[F(f_i \circ g)B \le
    A], \]
  which exactly means that all $f_i \circ g : p(Y,B) \to p(X_i,A)$ are
  $\mc F$-morphisms. Hence, $(\mc F,p)$ is a topological category, and
  we leave the readers to verify that the above two processes are
  mutually inverse.
\end{proof}

\begin{proposition}
  \label{prop:alltopocat}
  All the categories of semantics mentioned in
  Example~\ref{exam:examcat} are topological categories.
\end{proposition}
\begin{proof}[Proof Sketch.]
  It is evident that all the fibres of those mentioned examples are
  complete lattices. We only describe in each case how the pullback or
  pushforward maps are constructed, and trust the readers to verify
  the universal properties and functoriality. Given a function
  $f : X \to Y$:
  \begin{itemize}
  \item In $\Kr$, $f^*$ lifts a relation $R$ on $Y$ to the largest
    relation in $X$ such that $f$ is monotone, i.e. for any
    $x,x'\in X$, $(x,x')\in f^*R$ iff $(fx,f'x) \in R$. The pullback
    maps in $\Pre,\Eqv$ are inherited from $\Kr$.

  \item In $\Top$, the pullback $f^*$ maps a topology $\gamma$ on $Y$
    to the so-called weak topology on $X$, i.e. $U \in f^*\gamma$ iff
    there exists $V\in\gamma$ that $U = f\inv(V)$.

  \item In $\Nb$, the description of $f^*$ is similar to that in
    $\Top$. For a neighbourhood relation $F$ on $Y$, the lift $f^*F$
    satisfies that $(x,U) \in f^*F$ iff there exists $V\subseteq Y$
    that $U = f\inv(V)$ and $(fx,V) \in F$.

  \item In $\CABAO$, it is easier to describe the pushforward
    maps. Given any endo-function $m$ on $\wp(X)$, its pushforward is
    the operator $\forall_f \circ m \circ f\inv$ on $Y$, where
    $\forall_f$ is the right adjoint of $f\inv$.

  \item In $\Evl$, evidently the pullback $f^*$ is obtained by
    post-composing with $f\inv$. \qedhere
  \end{itemize}
\end{proof}

At this point, we have accomplished our first goal to recognise all
the instances of semantics in Example~\ref{exam:examcat} as
topological categories. We end this section by describing the
\emph{product} construction:

\begin{definition}[Product of Topological Categories]
  \label{def:producttopocat}
  For any family $\set{\mc A_i}_{i\in I}$ of topological categories
  viewed as functors $\set{\mc A_i : \Set\op \to \InfL}_{i\in I}$,
  their product $\prod_{i\in I}\mc A_i$ is given as the following
  composition,
  \[
    \begin{tikzcd}
      \Set\op \ar[r, "\prod_{i\in I}\mc A_i"] & \prod_{i\in I}\InfL
      \ar[r, "\bigoplus_{i\in I}"] & \InfL.
    \end{tikzcd}
  \]
\end{definition}

The functor $\bigoplus_{i\in I}$ is the biproduct functor on $\InfL$,
which takes a family of inflattices to its set-theoretic product with
entry-wise order. In other words, the fibre
$(\prod_{i\in I}\mc A_i)_X$ of a product is simply the product of the
fibres $\prod_{i\in I}(\mc A_i)_X$. It is easy to verify that
$\prod_{i\in I}\mc A_i$ is indeed their categorical product in the
category of concrete categories and concrete functors. The product
construction for instance allows us to combine a Kripke model with an
evaluation function by looking at $\Kr \times \Evl$, or to consider a
family of models by introducing $\mc A\mta$ for any set $\Sigma$,
which is the $\Sigma$-indexed product of $\mc A$ with itself.

\section{Interpreting Modalities via Geometric Data}
\label{sec:modality}

In this section, we will see how the categorical structure we have
described in Section~\ref{sec:preliminary} would unify the
\emph{interpretation of modalities} in each different types of
semantics. We start by briefly recalling the very basics of the modal
language and its interpretation; standard references include
\cite{blackburnrijkevenema2001,vanBenthem2010open}. Let a non-empty
set $\Sigma$ serve as the signature, and let $\ms P$ be a non-empty
set of propositional variables. The modal language $\mc L_\Sigma$ over
the signature $\Sigma$ and the variable set $\ms P$ is the smallest
set of formulas containing $\ms P$ and closed under forming
conjunctions, negations, and adding modalities $\Box_a$ for all
$a\in\Sigma$. When $\Sigma$ is a singleton, we will omit the
subscript, and $\mc L$ denotes the usual modal language with a single
modality. We will refer to it as the \emph{basic modal
  language}. Other logical connectives are viewed as defined notions.

In any set-based semantics of modal logic, the classical propositional
connectives are always interpreted by the Boolean operations on the
power set algebra. From an algebraic point of view, the interpretation
of the additional modality, in its most general form, should be given
by an arbitrary \emph{endo-function on the power set}, which is
exactly the structure of a CABAO. Hence, we define the structure of a
\emph{semantic functor} to provide the interpretation of basic modal
language:

\begin{definition}[Semantic Functor and Modal Category]
  \label{def:semanticfunctor}
  Let $(\mc A,\geo-)$ be a topological category. A \emph{semantic
    functor} on $\mc A$ is a concrete functor
  $(-)\upa : \mc A \to \CABAO$. A \emph{modal category} is then a
  topological category together with a semantic functor.
\end{definition}

For any modal category $\mc A$ with semantic functor $(-)\upa$, we
recursively define the interpretation of modal formulas as follows:
For any set $X$ and any pair $(A,V)$ in $(\mc A \times \Evl)_X$,
\[ \assv p_A = V(p), \quad \assv{\varphi\wedge\psi}_{A} =
  \assv{\varphi}_{A} \cap \assv{\psi}_A, \quad \assv{\neg\varphi}_{A}
  = X\backslash\assv{\varphi}_{A}, \quad \assv{\Box\varphi}_{A} =
  A\upa(\assv{\varphi}_{A}). \] We may also define the more familiar
\emph{local} version of semantics, and write $A,V,x \models \varphi$
whenever $x \in \assv{\varphi}_A$. Evidently, the identity functor on
$\CABAO$ establishes itself as a modal category. We see below that all
other categories of semantics mentioned previously have modal category
structures:

\begin{proposition}
  \label{prop:semanticfuncfullsubcat}
  There exist \emph{fully faithful} modal functors on
  $\Kr,\Pre,\Eqv,\Top$ and $\Nb$ that embeds them into $\CABAO$,
  inducing the usual semantics of modal logic.
\end{proposition}
\begin{proof}[Proof Sketch.]
  Again, we only describe the construction of semantic functors in
  each case, and trust the readers to verify their fully faithfulness:
  \begin{itemize}
  \item Recall for any relation $R \subseteq X \times Y$, there exists
    an induced operator $\forall_R : \wp(X) \to \wp(Y)$, such that for
    any $S\subseteq X$,
    $\forall_R(S) = \scomp{y\in Y}{\forall x[xRy \nt x\in S]}$. We
    then construct the embedding $\Kr \hook \CABAO$ by sending each
    relation $R$ in fibre $\Kr_X$ to the operator $\forall_{R\dg}$,
    where $R\dg$ is the dual relation of $R$. The semantic functors on
    $\Pre$ and $\Eqv$ are inherited from the one on $\Kr$.

  \item For $\Top$, it sends each topology $\tau$ on a set $X$ to the
    interior operator $j_\tau$ it induces.

  \item For $\Nb$, it assigns $E$ in $\Nb_X$ to $n_E$, such that
    $n_E(S) = \scomp{x}{(x,S)\in E}$ for any $S\subseteq X$. \qedhere
  \end{itemize}
\end{proof}

Proposition~\ref{prop:semanticfuncfullsubcat} then completes our
categorical unification of all the mentioned types of semantics on how
they interpret the basic modal language. Clearly, our approach of
given in Definition~\ref{def:semanticfunctor} closely relates to the
spirit of \emph{algebraic semantics} of modal logic.  But one
additional insight our categorical framework suggests is an even
closer connection between these different types of semantics with
modal algebras via Proposition~\ref{prop:semanticfuncfullsubcat}, in
that the single notion of continuous morphisms between CABAOs as
defined in Example~\ref{exam:examcat} explains all the different types
of morphisms in these topological categories, by identifying them as
\emph{full subcategories} of $\CABAO$.

Intuitively, it is precisely the semantic functor that provides the
interpretation of modalities in all cases, but we can establish the
correspondence in a more formal way, by considering
\emph{transformation of models} as mentioned in
Section~\ref{sec:intro}. We define when a concrete functor between two
modal categories interacts well with a specific fragment of modal
logic:

\begin{definition}[Preservation of Language]
  \label{def:preservelanguage}
  Let $\mc A,\mc B$ be two modal categories, which both support the
  interpretation of certain fragment of modal language $\mc L_0$ which
  extends $\mc L$. We say a concrete functor $F : \mc A \to \mc B$
  \emph{preserves the interpretation of the language $\mc L_0$}, if
  the following happens: For $(A,V)$ in $(\mc A \times \Evl)_X$ over
  some set $X$ and for any formula $\varphi \in \mc L_0$, we have
  $\assv{\varphi}_{A} = \assv{\varphi}_{FA}$.
\end{definition}

In other words, a concrete functor $F$ preserves the interpretation of
a language $\mc L_0$ iff the evaluation of each formula in $\mc L_0$
remains unchanged when we apply the transformation $F$. As a first
example of establishing an exact correspondence between a semantic
structure and a particular syntactic pattern, we prove the following
theorem:

\begin{theorem}
  \label{thm:modalfuncpreserveL}
  For any concrete functor $F : \mc A \to \mc B$ between two modal
  categories $(\mc A,(-)\upa_{\mc A})$ and $(\mc B,(-)\upa_{\mc B})$,
  it commutes with the two semantic functors iff it preserves the
  interpretation of $\mc L$.
\end{theorem}
\begin{proof}
  Suppose $F$ does not commute with the two semantic functors, then
  for some object $A$ in $\mc A$ over some set $X$, $(A)\upa_\mc A$
  and $(FA)\upa_{\mc B}$ would not agree. This means that the two
  operators on $\wp(X)$ do not coincide, which implies they must not
  coincide on some subset $S \subseteq X$. Consider the simple formula
  $\Box p$, and an evaluation function $V$ that assigns $p$ to $S$. By
  definition, $\assv{\Box p}_A$ and $\assv{\Box p}_{FA}$ will
  \emph{not} be the same.

  The proof of the only if direction is obviously by induction on the
  structure of formulas, and the only interesting case is the one
  involving modalities. Since $F$ is assumed to be a modal functor, we
  must have $(A)\upa_{\mc A} = (FA)\upa_{\mc B}$ for any $A$ in
  $\mc A$, which means that the interpretation of the modalities by
  $A$ through $(-)\upa_{\mc A}$ and by $FA$ through $(-)\upa_{\mc B}$
  are identical, which suffices for the inductive proof.
\end{proof}

Theorem~\ref{thm:modalfuncpreserveL} provides the precise formal
content of what we mean informally by the correspondence between the
syntactic structure of modalities and the semantic structure of
semantic functors of a modal category. And henceforth, we will refer
to those concrete functors between two modal categories which commutes
with the semantic functors on both sides as \emph{modal
functors}. There are already many interesting examples of modal
functors we can explore, and below we only list a few:

\begin{example}
  \label{exam:modalfunc}
  Here we list some interesting examples of model transformations
  between the modal categories we have introduced so far:
  \begin{itemize}
  \item By definition, any modal category has a unique modal functor
    mapping into $\CABAO$, which makes it the terminal object in the
    category of modal categories and modal functors.

  \item Since the semantic functors in $\Pre$ and $\Eqv$ are induced
    by the one in $\Kr$, the embeddings $\Eqv\hook\Pre$ and
    $\Pre\hook\Kr$ are both modal functors.

  \item There is a modal embedding $\Pre \hook \Top$, assigning a
    preorder its Alexandroff topology.

  \item In fact, we can show that $\Nb$ is isomorphic to $\CABAO$,
    which means that all the above examples has a modal embedding into
    $\Nb$ as well.
  \end{itemize}
  It is also instructive to look at counter-examples of modal
  functors. It turns out, the above modal embeddings all have either a
  left or a right adjoint, and these adjoints are usually \emph{not}
  modal embeddings with respect to the semantic functors we have
  constructed in Proposition~\ref{prop:semanticfuncfullsubcat}:
  \begin{itemize}
  \item We have both a left and a right adjoint
    $\Pre\rightrightarrows\Eqv$ for the modal embedding
    $\Eqv \hook \Pre$, sending a preorder to the smallest equivalence
    relation containing it and the least one it contains. These
    adjoints do not commute with the semantic functors since they
    change the relation. Similarly, there is a left adjoint
    $\Kr \to \Pre$ sending a relation to its preorder closure, which
    isn't modal either.

  \item The embedding $\Pre\hook\Top$ has a right adjoint
    $\Top \to \Pre$, sending a topological space to its specialisation
    order, but this construction does not preserve the information of
    all open neighbourhoods of a point, hence it is also not
    modal. \qedhere
  \end{itemize}
\end{example}

However, the mere syntactic structure of a modality, arguably, has not
too much to do with the rich structure of topological categories we
have seen in Section~\ref{sec:preliminary}. In fact, the notion of
semantic functors and modal categories in
Definition~\ref{def:semanticfunctor} can indeed be stated more
generally for concrete categories, not only for topological ones. The
true usage of the full structure of topological categories emerges
when we consider further syntactic extensions of modal logic, which
are the topics of the next two sections.

\section{Modal Strength, Group Knowledge and Fibre Structure}
\label{sec:group}
In this section, we will proceed to study the extension of multi-agent
fragment of modal logic, with explicit syntactic comparison of
\emph{modal strength}, or \emph{dependence relation}, between
different modalities, and forming \emph{group agents}. Recent
works~\cite{baltag2021simple,baltag2021lcd} put dependence purely in
modal terms, but they have only considered the relational and
topological contexts. How to form group agents is also an active topic
for current research on modal logic and collective
agency~\cite{goble2006deontic,tamminga2021expressivity}, but almost
all approaches focus on a single type of models. In both cases, our
categorical approach allows a unifying description for all types of
semantics, which is one of the main benefit. Our ultimate goal is
again to identify an exact correspondence between these syntactic
patterns with certain semantic structures of topological categories,
with formal content similar to that of
Theorem~\ref{thm:modalfuncpreserveL}.

Let's first look at the simple extension of a multi-modal language,
i.e. when the indexed set $\Sigma$ is not a singleton. There will be
different modalities $\Box_a,\Box_b,\cdots$ with $a,b\in\Sigma$ in the
language $\mc L_\Sigma$. It should be straight forward to recognise
that the multi-agent fragment $\mc L_\Sigma$ are related to taking the
\emph{products} of topological categories. Given any modal category
$\mc A$, recall that we use $\mc A\mta$ to denote the $\Sigma$-indexed
self-product of $\mc A$. Any semantic functor $(-)\upa$ on $\mc A$
naturally extends to one from $\mc A\mta$ to $\CABAO\mta$, which by an
abuse of notation we also denote as $(-)\upa$: Given any object
$(A_a)_{a\in\Sigma}$ in the fibre $\mc A\mta_X$, which by our
construction in Definition~\ref{def:producttopocat} is simply a
$\Sigma$-indexed tuple of objects in the fibre $\mc A_X$, we have
$(A_a)_{a\in\Sigma}\upa = (A_a\upa)_{a\in\Sigma}$. The
$\Sigma$-indexed tuple $(A_a\upa)_{a\in\Sigma}$ is then expected to
provide the interpretation of each modality $\Box_a$ in the language
$\mc L_\Sigma$ for any $a\in\Sigma$, using the corresponding object
$A_a\upa$. Intuitively, different modalities correspond to different
objects in the same fibre of a topological category. Hence, given any
$((A_a)_{a\in\Sigma},V)$ in the fibre $(\mc A\mta\times\Evl)_X$, we
may change the clause of modalities in the recursive definition of
evaluation of formulas to
$\assv{\Box_a\varphi}_{(A_a)_{a\in\Sigma}} =
(A_a)\upa(\assv\varphi_{(A_a)_{a\in\Sigma}})$, to interpret
$\mc L_\Sigma$.

However, in the language $\mc L_\Sigma$, we treat different modalities
as different individuals, and do not consider the possible relations
between different modalities. But we do have a meaningful way
comparing them, since semantically they denote different objects
within the same fibre of a topological category $\mc A$, and there is
a canonical order in each fibre $\mc A_X$. It turns out, this partial
order within each fibre signifies the \emph{modal strength} of
different modalities. Explicitly, suppose we have two objects $A,B$ in
the fibre $\mc A_X$ that $A \le B$. The semantic functor then gives us
two operators $m_A \le m_B$ in $\CABAO_X$, which, according to our
definition of morphisms in $\CABAO$, actually means
$m_B \subseteq m_A$.

In different contexts, the modal strength relation has various
incarnations. For instance, in epistemic or doxastic logic, we read
the modal formula $\Box_a\varphi$ as agent-$a$ knows or believes
$\varphi$ (cf.~\cite{blackburn2006handbook}). Now if we have
$A_a \le A_b$ in the fibre $\mc A_X$, the above induced two modalities
satisfying $m_b \subseteq m_a$ would actually suggest that there is an
\emph{epistemic dependence} between the two agents' knowledge or
belief: Whenever $b$ knows some proposition at state $x \in X$,
viz. $x \in m_b(\ass\varphi)$, $a$ also knows it at that state,
because $x \in m_b(\ass\varphi) \subseteq m_a(\ass\varphi)$. In other
applications, such modal strength comparison would mean something
else.

This observation motivates us to add such comparison of modalities
explicitly into our syntax, in the form of \emph{dependence
  atoms}. For any $a,b\in\Sigma$, we could add an atomic proposition
$K_ab$ into our language, with the intuitive reading of $K_ab$ as
stating the modality denoted by $a$ lies below the one denoted by
$b$. We refer to this extended language as $\mc L^D_\Sigma$. But to
interpret such dependence atoms as predicates, we need the following
\emph{local} version of strength orders between two operators on the
same power set algebra:

\begin{definition}
  \label{def:localdependence}
  For any two operators $m,n$ in $\CABAO_X$ and any $U \subseteq X$,
  we say $m$ \emph{locally depends on $n$ in $U$}, denoted as
  $m \subseteq_U n$, if for any $S \subseteq X$ and any $x\in U$,
  $x \in m(S) \nt x \in n(S)$.
\end{definition}
\noindent
In this way, the global relation $m \subseteq n$ is the same as
$m \subseteq_X n$. When $U$ is a singleton $\set{x}$, we simply write
$m \subseteq_x n$. The following observation is crucial for us to
define the interpretation of the dependence atoms:

\begin{lemma}
  For any $m,n$ in $\CABAO_X$, there is a maximal subset $U$ that
  $m \subseteq_U n$.
\end{lemma}
\begin{proof}
  By definition, for the empty set $\emptyset$ we always have
  $m \subseteq_\emptyset n$, since the universal quantification
  $\forall x\in\emptyset$ is vacuous. Furthermore, local dependence is
  closed under taking unions, since it is trivial to note that
  $m \subseteq_{\bigcup_{i\in I} U_i}$ iff for any $i\in I$,
  $m \subseteq_{U_i} n$. Thus, the maximal subset $U$ is given by
  $\scomp{x}{m \subseteq_x n}$.
\end{proof}

Given an object $(A_a)_{a\in\Sigma}$ in the fibre $\mc A\mta_X$, the
interpretation $\ass{K_ab}$ of the newly added dependence atoms should
now be defined as the maximal subset $U$ of $X$, such that
$A_b\upa \subseteq_U A_a\upa$ holds. This is exactly how the
dependence atoms are interpreted in any topological categories. We
might also give the local version of the truth condition,
$(A_a)_{a\in\Sigma},x \models K_ab$ iff $A_b\upa \subseteq_x
A_a\upa$. Notice that the interpretation of $K_ab$ is
\emph{independent} from the choice of the evaluation function $V$ on
$X$. We may look at the concrete meaning of such dependences in all
the remaining examples we have considered so far:

\begin{example}
  We list here how local dependence looks like in each exemplar modal
  category:
  \begin{itemize}
  \item In $\Kr$, $\Pre$ and $\Eqv$, given relations
    $(R_a)_{a\in\Sigma}$ on $X$, we have
    $(R_a)_{a\in\Sigma},x \models K_ab$ iff $R_a[x] \subseteq
    R_b[x]$. In the epistemological interpretation, this means
    agent-$a$'s uncertainty locally at $x$ is less than $b$'s.

  \item In $\Top$, given topologies $(\tau_a)_{a\in\Sigma}$ on $X$,
    $(\tau_a)_{a\in\Sigma},x \models K_ab$ iff
    $1_X : (X,\tau_a) \to (X,\tau_b)$ is \emph{locally continuous} at
    $x$. This relates to the continuity view of epistemic dependence
    discussed in~\cite{baltag2021lcd}.

  \item In $\Nb$, given neighbourhoods $(E_a)_{a\in\Sigma}$ on $X$,
    $(E_a)_{a\in\Sigma},x \models K_ab$ iff $E_b[x] \subseteq
    E_a[x]$. In evidence based logic, this interprets as the evidence
    set of $b$'s is contained in that of $a$'s locally at $x$
    (cf.~\cite{vanbenthem2012evidence}). \qedhere
  \end{itemize}
\end{example}

In this case, preserving the interpretation of the multi-agent
modalities and the dependence atoms does not require anything else
than being a modal functor:

\begin{theorem}
  \label{thm:conservextenDSigma}
  For any concrete functor $F : \mc A \to \mc B$ between two modal
  categories, it preserves the interpretation of $\mc L^D_\Sigma$ iff
  it preserves the interpretation of $\mc L$.
\end{theorem}
\begin{proof}
  The only if part is trivial, since $\mc L^D_\Sigma$ is an extension
  of $\mc L$. For the if part, by Theorem~\ref{thm:modalfuncpreserveL}
  we know $F$ must be a modal functor. This implies that for any
  $a\in\Sigma$ and any tuple $(A_a)_{a\in\Sigma}$ in $\mc A\mta,$ we
  must have $(A_a)\upa_{\mc A} = (FA_a)\upa_{\mc B}$, which means that
  $A_a$ induces the same operator as $FA_a$. This suffices for the
  preservation of the fragment $\mc L_\Sigma$ by $F$. $F$ preserving
  dependence atoms is also immediate, since their interpretation only
  relies on the operators on the underlying set.
\end{proof}

However, this changes once we start to combine sets of agents into a
single agent and consider such group structures explicitly in our
syntax. From a philosophical perspective, when modelling the inference
and reasoning patterns of agents under certain information structure
using modal logic, we not only care about individual agents
themselves, but we would also like to study how \emph{a group of
  agents} as a whole reasons and interacts with each other. As
mentioned, this is an active topic on how to represent group agency in
different contexts. Most of the traditional developments of group
agency in modal logic are based on Kripkean
semantics~\cite{vanBenthem2010open,vanbenthem2011ldii}, but there has
been recent efforts exploring how to define common knowledge of a
group in topological semantics~\cite{topoapproach2019}. Again, our
categorical approach would uniformly describe the group structure in
any topological category associated with every type of semantics.

To combine a group of agents to a single one, it requires us to
transform an object in $\mc A^G$ for any subset $G \subseteq \Sigma$,
which is a tuple representing each individual agent in the group $G$,
to a single object in $\mc A$, which corresponds to the collective
group agent. Naturally, there are two canonical ways to do this in
general for any set $G$, using the fact that each fibre in a
topological category is not only a poset, but indeed a \emph{complete
  lattice}. In particular, we can form two (families of) concrete
functors $\bigwedge,\bigvee : \mc A^G \to \mc A$. As the symbols
suggest, for any tuple $(A_a)_{a\in G}$ in $\mc A^G$, they act on it
as follows: $\bigwedge (A_a)_{a\in G} = \bigwedge_{a\in G} A_a$, and
$\bigvee (A_a)_{a\in G} = \bigvee_{a\in G} A_a$. Functoriality of
$\bigwedge,\bigvee$ should be immediate.

These functors then allow us to combine a group of agents of arbitrary
size into a single one. We will denote them as the $\bigwedge$- and
$\bigvee$-combination of group agents, and they correspond to two
different readings of what a group of agents means. Intuitively, the
$\bigwedge$-combination means the group \emph{shares the information}
of each individual, as if they are \emph{physically together}. Because
once we form a group $\bigwedge_{a\in G}A_a$, for any individual $a$
in the group $G$ we would have $\bigwedge_{a\in G}A_a \le A_a$ in the
fibre, which implies $A_a\upa \subseteq (\bigwedge_{a\in
  G}A_a)\upa$. Just as we have discussed before, if we adopt an
epistemic interpretation of modalities, this means that whatever agent
$a$ knows, so does the group, and this holds for any agent in this
group. Furthermore, the meet taken in the fibre $\mc A_X$ actually
shows that the group modelled by $\bigwedge_{a\in\Sigma}A_a$ is the
universal one that has this property. This informally suggests that
the group acts like an agent who has access to all the information
owned by each individual agent in this group, exactly like the case
when everyone in the group has come to a single location, and put all
of their information on the table where anyone can see. In $\Pre$ or
$\Eqv$, the $\bigwedge$-combination simply take the conjunction of all
the relations, and this is exact the well-known \emph{distributive
  knowledge} of a group (cf.~\cite{vanbenthem2011ldii}). Hence, the
$\bigwedge$-combination generalises distributive knowledge to all
types of semantics.

On the other hand, the $\bigvee$-combination means the group
\emph{shares the uncertainties} of each individual, as if they are
only \emph{abstractly} considered as a single agent. Dual to the case
before, we must have $(\bigvee_{a\in G}A_a)\upa \subseteq A_a$ for any
$a\in G$. This implies that for the combined group, if it knows
something then necessarily each individual in the group also knows
this, and the group agent is the universal one that has this
property. To better compare with the existing literature, we observe
the following simple result:

\begin{lemma}
  \label{lem:commonknowledge}
  If the semantic functor $(-)\upa$ on a topological category $\mc A$
  always induces monotone and idempotent operators, then
  $(\bigvee_{a\in G}A_a)\upa \subseteq A_{a_1}\upa \circ \cdots \circ
  A_{a_n}\upa$ for any $a_1,\cdots,a_n\in G$.
\end{lemma}
\begin{proof}
  If follows by $(\bigvee_{a\in G}A_a)\upa \subseteq A_{a_i}\upa$ for
  any $i$, and monotonicity, idempotence of these operators.
\end{proof}
\noindent
Translating back to natural language, in the condition of
Lemma~\ref{lem:commonknowledge}, what the $\bigvee$-combined group
knows is much more restrictive, in that if the group knows something,
then any agent in the group also knows it, and furthermore $a_i$ knows
that $a_j$ knows that $\cdots$ that $a_k$ knows it. This shows that
the $\bigvee$-combination is a generalisation of the \emph{common
  knowledge} of a group (again, cf.~\cite{vanbenthem2011ldii}).

We may now formally define the syntactic extension where we also allow
group formation in our logic. For any indexed set $\Sigma$, we let
$\Sigma_l$, $\Sigma_r$ be \emph{synonyms} for the power set
$\wp(\Sigma)$. The language $\mc L^D_{\Sigma_l}$ and
$\mc L^D_{\Sigma_r}$ is nothing more but the modal languages with
agent symbols in $\Sigma_l,\Sigma_r$, respectively, together with all
the dependence atoms between these group agents. However, we write in
this way because to interpret the language $\mc L^D_{\Sigma_l}$ or
$\mc L^D_{\Sigma_r}$, we still only need to work within $\mc A\mta$,
not $\mc A^{\Sigma_l}$ or $\mc A^{\Sigma_r}$.

Given an object $(A_a)_{a\in\Sigma}$ in $\mc A\mta$ over the set $X$,
we can interpret the modal operators for a group of agents in the two
fragments as either the $\bigwedge$- or $\bigvee$-combination. For any
subset $G \subseteq \Sigma$, we define the interpretation of $\Box_G$
in $\mc L^D_{\Sigma_l}$ as the operator
$\left(\bigwedge_{a\in G} A_a \right)\upa$; and similarly for the
language $\mc L^D_{\Sigma_r}$, $\Box_G$ is interpreted as the operator
$\left(\bigvee_{a\in G} A_a \right)\upa$. Building on what we have
developed before, this suffices to interpret the two languages
$\mc L^D_{\Sigma_l}$ and $\mc L^D_{\Sigma_r}$. Of course, for a
singleton group $\set a$, its interpretation under the two fragments
coincide, which still corresponds to the usual interpretation of the
operator $A_a\upa$. The upshot is that we can identify the following
valid logical rules in the two fragments $\mc L^D_{\Sigma_l}$ and
$\mc L^D_{\Sigma_r}$:

\begin{proposition}
  For any modal category $\mc A$, the following axioms are valid in
  $\mc L_{\Sigma_l}^D$ (resp. $\mc L^D_{\Sigma_r}$):\footnote{Half of
    these axioms corresponding to the fragment $\mc L^D_{\Sigma_l}$
    has already been identified
    in~\cite{baltag2021simple,baltag2021lcd} in the special case of
    $\Top$.}
  \begin{itemize}
  \item \textbf{\emph{Inclusion}}: $K_{G}H$ (resp. $K_HG$), provided
    $H \subseteq G$;
  \item \textbf{\emph{Additivity}}:
    $K_{G}H \wedge K_GP \to K_G(H\cup P)$ (resp.
    $K_{H}G \wedge K_PG \to K_{H\cup P}G$);
  \item \textbf{\emph{Transitivity}}: $K_GH \wedge K_HP \to K_GP$
    (resp. $K_GH \wedge K_HP \to K_GP$);
  \item \textbf{\emph{Transfer}}:
    $K_GH \wedge \Box_H\varphi \to \Box_G\varphi$ (resp.
    $K_GH \wedge \Box_H\varphi \to \Box_G\varphi$).
  \end{itemize}
\end{proposition}
\begin{proof}
  We only prove the case for $\mc L^D_{\Sigma_l}$; the other case is
  completely dual. Let $(A_a)_{a\in\Sigma}$ be any object in
  $\mc A\mta$ over $X$. Whenever we have
  $H \subseteq G \subseteq \Sigma$, we have
  $A_G = \bigwedge_{a\in G} A_a \subseteq \bigwedge_{a\in H} A_a =
  A_H$, which implies $A_H\upa \subseteq A_G\upa$. Hence, according to
  our definition of the interpretation of the dependence atoms, we
  have $\ass{K_GH} = X$, and this validates \textbf{Inclusion}. For
  any two groups $H,P$, by definition
  $A_{H\cup P} = \bigwedge_{a\in H\cup P} A_a = A_H \wedge A_P$, which
  implies $m_H \cup m_P \subseteq m_{H\cup P}$. Now locally, suppose
  for some $x\in X$ we have $x\in\ass{K_GH}$ and
  $x\in\ass{K_GP}$. Then for any $S\subseteq X$,
  $x \in m_{H\cup P}(S) \nt x \in m_H(S) \cup m_P(S)$. Either
  $x\in m_H(S)$ or $x\in m_P(S)$, we would have $x \in m_G(S)$,
  according to our assumption that $K_GH$ and $K_GP$ locally holds at
  $x$. Hence, the \textbf{Additivity} law also holds. The validity of
  \textbf{Transitivity} and \textbf{Transfer} axioms are evident.
\end{proof}

Up to this point, we have completed our generalisation of group
structure to all the exemplar modal categories in a uniform way, and
identified a set of valid inference rules. The remaining task is then
to identify which part of the semantic structure in topological
categories does the syntactic group-forming operation corresponds
to. Considering our usage of the complete lattice structure of fibres,
the following result should be of no surprise:

\begin{theorem}
  Let $F : \mc A \to \mc B$ be a modal functor between two modal
  categories, and suppose the semantic functor $(-)\upa_{\mc B}$ is
  injective on objects. $F$ preserves arbitrary meets (resp. joins)
  fibre-wise, i.e. the induced functions $F_X : \mc A_X \to \mc B_X$
  on fibres is a morphism in $\InfL$ (resp. $\SupL$) for any set $X$,
  iff it preserves the interpretation of the language
  $\mc L^D_{\Sigma_l}$ (resp. $\mc L^D_{\Sigma_r}$) for any indexed
  set $\Sigma$.
\end{theorem}
\begin{proof}
  Again, we only prove the case for $F$ preserving meets fibre-wise
  and the preservation of the interpretation of $\mc
  L^D_{\Sigma_l}$. We already know from
  Theorem~\ref{thm:conservextenDSigma} that $F$ is a modal functor iff
  it preserves the interpretation of $\mc L^D_\Sigma$, thus it
  suffices to show it further preserves the interpretation of
  $\bigwedge$-group-formation iff it preserves meets fibre-wise. From
  how the $\bigwedge$-group modality is defined, it is immediate to
  note that $F$ preserves the interpretation of $\mc L^D_{\Sigma_l}$
  iff $\left(\bigwedge_{a\in\Sigma} A_a \right)\upa_{\mc A}$, which by
  the fact of $F$ being a modal functor is equal to
  $\left(F\bigwedge_{a\in\Sigma} A_a\right)\upa_{\mc B}$, coincides
  with $\left(\bigwedge_{a\in\Sigma} FA_a\right)\upa_{\mc B}$, for any
  $(A_a)_{a\in\Sigma}$. By assumption on $(-)\upa_{\mc B}$, this holds
  iff $F\bigwedge_{a\in\Sigma} A_a = \bigwedge_{a\in\Sigma} FA_a$,
  which exactly means $F$ preserves meets fibre-wise.
\end{proof}

Consider the various model transformations we have described in
Example~\ref{exam:modalfunc}, Theorem~\ref{thm:conservextenDSigma}
immediately tells us how these functors behave with respect to group
knowledge. For instance, since the modal embedding $\Eqv \hook \Pre$
has both a concrete left and right adjoint, it must preserve both
meets and joins fibre-wise, which suggests that the two fragments
$\mc L^D_{\Sigma_l}$ and $\mc L^D_{\Sigma_r}$ behave coherently
between $\Eqv$ and $\Pre$. However, as for the embedding of $\Eqv$ and
$\Pre$ into $\Kr$, it only has a concrete left adjoint but lacks a
right one, which means only the $\bigwedge$-group formation, viz. the
distributive knowledge, coincide in $\Eqv,\Pre$ and $\Kr$, but not the
common knowledge. We can see this more explicitly, since the join in
fibres of $\Kr$ are simply unions of relations, while in $\Pre$ and
$\Eqv$ we must further take the transitive closure of unions of
relations. Other modal embeddings can be analysed in a similar
fashion.

\section{Logical Dynamics and Fibre Connections}
\label{sec:dynamics}

The ``dynamic turn'' of modal logic in the recent two decades makes
\emph{logical dynamics} another very important topic in the current
literature. In this section, we will see how certain general types of
logical dynamics could be subsumed into our categorical framework in a
similar fashion as before.

Logical dynamics concerns with the reasoning patterns of agents when
\emph{new information} comes in, which generally \emph{changes the
  underlying set} of a model. This is where the \emph{fibre
  connection} plays a crucial role, because it allows us to transfer
the geometric data over the original model to the updated model in a
uniform way. For simplicity, below we describe all the dynamic
extensions based on the simplest fragment $\mc L$, but it should be
clear that our method can be equally applied to other fragments as
well.

To warm up, we start by generalising the simplest form of dynamic
logic, known as PAL, public announcement logic
(cf.~\cite{plaza1989logics,plaza2007pal}). It concerns with
information events of publicly announcing that $\varphi$ holds, which
we denote as $!\varphi$. A typical formula in PAL is of the form
$[!\varphi]\psi$, intuitively read as $\psi$ is true after announcing
$\varphi$. For a modal category $\mc A$, given any object $(A,V)$ in
the fibre $(\mc A \times \Evl)_X$, the information event $!\varphi$
naturally restricts the domain $X$ to the subset
$S = \assv{\varphi}_A$. If we denote the inclusion function
$S \hook X$ as $i$, then the natural way to transfer the geometric
data on $X$ to $S$ is by pulling back along $i$. This way, we obtain a
new semantic model $(i^*A,i^*V)$ over $S$, and the formula following
the dynamic operator $[!\varphi]$ could be interpreted in this new
model. We also need to transfer subsets of $S$ back to subsets of $X$,
to maintain the recursive structure of adding dynamic operators within
the syntax. The natural candidates are $\exists_i$ and $\forall_i$,
which we will see correspond to the pair of dual operators
$\pair{!\varphi}$ and $[!\varphi]$.

More formally, we define the extension $\mc L\PAL$ of $\mc L$ by the
smallest set of formulas containing $\mc L$ and is closed under
forming dynamic formulas of the form $[!\varphi]\psi$, with
$\varphi,\psi$ in $\mc L\PAL$. Following the above informal idea, we
define the interpretation of formulas in $\mc L\PAL$ by adding the
following recursive clause: For $(A,V)$ in $(\mc A \times \Evl)_X$, we
define $\assv{[!\varphi]\psi}_A = \forall_i\ass{\psi}^{i^*V}_{i^*A}$,
and
$\assv{\pair{!\varphi}\psi}_A = \exists_i\ass{\psi}^{i^*V}_{i^*A}$,
where $i$ is the inclusiong map $\assv{\varphi}_A \hook X$. Perhaps
the more familiar form of truth conditions of these dynamic operators
are the following equivalent local formulation: For any $x\in X$,
\begin{align*}
  A,V,x \models [!\varphi]\psi
  &\eff A,V,x \models \varphi \text{ implies } i^*A,i^*V,x \models \psi, \\
  A,V,x \models
  \pair{!\varphi}\psi
  &\eff A,V,x \models \varphi \text{ and } i^*A,i^*V,x \models \psi.
\end{align*}
Again, if we combine this general form of semantics of PAL in any
modal category with the special description of pullback maps in $\Kr$
given in the proof of Proposition~\ref{prop:alltopocat}, we recover
exactly the usual PAL dynamics developed for Kripke models, but we
also get the PAL dynamics in other types of semantics at the same
time. This again exhibits the usefulness of a unifying description of
semantics of modal logic.

Expectedly, the syntactic PAL dynamic operators in the $\mc L\PAL$
fragment should correspond to the semantic structure of initial lifts
along inclusions in a topological category:

\begin{theorem}
  Let $F : \mc A \to \mc B$ be a modal functor between two modal
  categories, and suppose the semantic functor $(-)\upa_{\mc B}$ is
  injective on objects. $F$ further preserves the interpretation of
  $\mc L\PAL$ iff it preserves the initial lifts of any injections,
  i.e. for any inclusion map $i : S \hook X$ and for any object $A$ in
  the fibre $\mc A_X$, $Fi^*A = i^*FA$ holds.
\end{theorem}
\begin{proof}
  Again for the if direction we prove by induction, and the only case
  we need to think about is for the PAL dynamic operator. Given
  $\varphi,\psi$ and any $(A,V)$ in $(\mc A \times \Evl)_X$, by
  induction hypothesis we have
  $\assv{\varphi}_A = \assv{\varphi}_{FA}$, and we denote the
  inclusion map of this subset into $X$ by $i$. Now by definition of
  the interpretation of $[!\varphi]\psi$, we have
  $\assv{[!\varphi]\psi}_A = \forall_i\ass{\psi}^{i^*V}_{i^*A} =
  \forall_i\ass{\psi}^{i^*V}_{Fi^*A} =
  \forall_i\ass{\psi}^{i^*V}_{i^*FA} =
  \assv{[!\varphi]\psi}_{FA}$. Thus, $F$ preserves the interpretation
  of $\mc L\PAL$.

  On the other hand, suppose for some object $A$ in the fibre
  $\mc A_X$ and for some injection $i : S \hook X$, we have $i^*FA$ is
  \emph{not} equal to $Fi^*A$. This in particular suggests that the
  associated operators $(i^*FA)\upa_{\mc B}$ and $(Fi^*A)\upa_{\mc B}$
  on $S$ are not identical, and they must disagree at some subset $T$
  of $S$. Then let $V$ be an interpretation on $X$ such that
  $V(p) = S$ and $V(q) = T$. Consider the interpretation of the
  formula $\pair{!p}\Box q$. On one hand, we have
  $\assv{\pair{!p}\Box q}_A = \exists_i\ass{\Box q}^{i^*V}_{i^*A} =
  \exists_i\ass{\Box q}^{i^*V}_{Fi^*A} = (Fi^*A)\upa_{\mc B}(T)$. On
  the other hand, we have
  $\assv{\pair{!p}\Box q}_{FA} = \exists_i\ass{\Box q}^{i^*V}_{i^*FA}
  = (i^*FA)\upa_{\mc B}(T)$. By assumption, $(Fi^*A)\upa_{\mc B}(T)$
  does not coincide with $(i^*FA)\upa_{\mc B}(T)$, and thus $F$ does
  not preserve the interpretation of $\mc L\PAL$ by definition.
\end{proof}

For those model transformations that has a concrete left adjoint, they
automatically commutes with all pullback maps, hence preserves the
interpretation of $\mc L\PAL$. Perhaps surprisingly, all of the modal
embeddings described in Example~\ref{exam:modalfunc} actually do
commutes with pullbacks of \emph{injections}, though not all of them
have a concrete left adjoint, and this statement for arbitrary
functions is false. As a result, $\mc L\PAL$ is a particularly nice
fragment of dynamic logic to work with.

However, PAL as dynamic logic is still too restrictive. A much more
powerful dynamic mechanism is \emph{product update} in DEL, dynamic
epistemic logic~\cite{van2007dynamic,liu2011reasoning}. In product
update, information events themselves form a model $E$, which carries
additional geometric data signifying agent's uncertainly about which
event actually happens, and the update is parametrised by $E$. Each
event $e\in E$ is also equipped with a formula $\varphi_e$ that
specifies the precondition of that event happens. For any model over a
set $X$, the updated model is a subset of the product space
$E \times X$, consisting of those pairs $(e,x)$ where $x$ satisfies
the precondition of $e$. The geometric data over the updated set takes
into account the ones on both $X$ and $E$.

There are already several categorical reformulation and generalisation
of DEL in the literature,
e.g. see~\cite{kishida2017categories,cina2017thesis}, but most of them
are based on relational semantics, while our approach applies to
arbitrary topological categories. We first define the notion of a
\emph{product type}, which generalises event models:

\begin{definition}[Product Type]
  A \emph{product type} $\ms E$ for the modal category $\mc A$ is a
  tuple $\pair{E,B,W,\set{\psi_e}_{e\in E}}$, where $E$ is a set, and
  $B,W$ are objects in the fibre $\mc A_E,\Evl_E$. The family
  $\set{\psi_e}_{e\in E}$ is an $E$-indexed family of formulas within
  the language $\mc L$.
\end{definition}
\noindent
The notion of \emph{product type update} we are going to describe,
which generalises DEL, is parametrised by such a product type $\ms
E$. For any semantic model $(A,V)$ in the fibre
$(\mc A \times \Evl)_X$, we write $E \otimes_V X$ as the underlying
set of the updated model, which is given by the dependent sum
$\sum_{e\in E}\assv{\psi_e}_A$. Intuitively, the updated model is
indexed by events in $E$, whose fibre over $e$ is the set of all
possible words satisfying the precondition $\psi_e$. There are then
two natural projection maps $\pi_X : E \otimes_V X \to X$ and
$\pi_E : E \otimes_V X \to E$, and we define the geometric data
$(\ms E \otimes_V A,W \otimes V)$ in the fibre
$(\mc A \times \Evl)_{E \otimes_V X}$ to be $\pi_X^*A \wedge \pi_E^*B$
and $\pi_X^*V \wedge \pi_E^*W$, respectively. A typical dynamic
formula in product type update is of the form $[\ms E,S]\varphi$ or
$\pair{\ms E,S}\varphi$, where $\ms E$ is a product type and $S$ is a
subset of $E$. We define their interpretation as follows,
\[ \assv{[\ms E,S]\Phi}_A := \forall_{\pi_X}\left((S \otimes_V
    X)\to\ass{\Phi}^{W \otimes V}_{E\otimes_VA} \right), \quad
  \assv{\pair{\ms E,S}\Phi}_A := \exists_{\pi_X}\left((S \otimes_V
    X)\cap\ass{\Phi}^{W \otimes V}_{E\otimes_VA} \right), \]
where the set $S \otimes_V X = \sum_{e\in S}\assv{\psi_e}_A$ is a
subset of $E \otimes_V X$, and $\to$, $\cap$ are calculated in the
power set $\wp(E \otimes_V X)$. Again, interpreted our general
construction back in the relational context $\Kr$ of Kripke models,
one immediately recovers the usual product update in DEL.\footnote{In
  the literature, only the case where $S$ is a singleton set $\set e$
  is usually considered, but this is a minor generalisation. It is
  possible to define product update more generally along any function
  mapping into $E$, but we leave that for future work.}

In a word, the way we associate the geometric data on the updated
model $E \otimes_V X$ is by pulling back the ones over $X$ and $E$
along the two projection maps, and then take their intersection in the
fibre. However, a categorically minded reader would perhaps wonder
what happens to the degenerate case where we have an \emph{empty
  intersection}. Though being kind of trivial, this is in fact
important for correspondence results of product type update, which
will be stated later. Hence, we also introduce \emph{empty product
  update}, whose syntactic structure is extremely simple: It is of the
form $\ms U\varphi$, and for any $(A,V)$ in $(\mc A\times\Evl)_X$ we
define $\assv{\ms U\varphi}_{A}$ to be $\assv{\varphi}_{\top_X}$,
where $\top_X$ is the maximal element in $\mc A_X$. This is indeed a
form of dynamics, since the operators $\ms U$ results in the change of
the geometric data, though the update is constant in all cases. We
then define $\mc L\PRO$ to be the least fragment containing $\mc L$,
which is also closed under taking dynamic formulas of empty product
update and product type update.

Now that product type update is properly generalised to arbitrary
topological categories, we can realise PAL dynamics as special case of
product type update. In fact, for any formula $\varphi$, we can
associate it with a product type, which we also denote as
$!\varphi$. Explicitly, $!\varphi$ is the tuple
$\pair{1,\top,\top,\set{\varphi}}$, where $1$ is the singleton set,
and $\varphi$ is the corresponding precondition of the single element
in $1$. It is evident that the updated model by this product type
$!\varphi$ is exactly the one obtained by publicly announcing
$\varphi$ in PAL dynamics. In fact, many other types of dynamics turn
out to be special cases (cf.~\cite{van2007dynamic}).

Now, it should certainly be expected that the dynamic extension
$\mc L\PRO$ corresponds exactly to pullback maps between fibres and
finite meets within fibres. However, for product type update, we need
to slightly modify our definition of preservation of languages, since
now in the syntax of $\mc L\PRO$, we have explicitly included certain
semantic data, viz. the product types $\ms E$. We now say a concrete
functor $F : \mc A \to \mc B$ preserves the interpretation of
$\mc L\PRO$ if, after uniformly changing every product type
$\ms E = \pair{E,B,W,\set{\psi_e}_{e\in E}}$ appearing in the syntax
to $F\ms E = \pair{E,FB,W,\set{\psi_e}_{e\in E}}$, the resulting
interpretation remains unchanged under transformation of models
induced by $F$.\footnote{A far more general approach is to look at the
  relationship between a model transformation induced by a concrete
  functor $F$, and a particular syntactic translation $T$. Our notion
  of preservation of languages is then a special case when $T$ is the
  identity translation, or in the case of product type updates,
  translating $\ms E$ to $F\ms E$. We leave this for future
  investigation.} We then have the following result:

\begin{theorem}
  \label{thm:preserveproducttypeupdate}
  Let $F : \mc A \to \mc B$ be a modal functor between two modal
  categories, and suppose the semantic functor $(-)\upa_{\mc B}$ is
  injective on objects. $F$ further preserves the interpretation of
  $\mc L\PRO$ iff it preserves pullback maps and fibre-wise finite
  meets.
\end{theorem}
\begin{proof}
  The if part can be proven by a straight forward induction on the
  complexity of formulas in $\mc L\PRO$. The only if part is
  technically trickier, though the general idea is no different from
  previous proofs of such correspondence results. We include a
  detailed proof in Appendix~\ref{app} for the convenience of
  referees.
\end{proof}

\section{Conclusion}
In this paper, we have used the language of topological categories to
provide a unifying description of different types of semantics of
modal logic, and have showed how various semantic structures within
topological categories enable us to interpret different extensions of
modal logic, including modal strength, group structure, and logical
dynamics. We believe our approach is instructive for the current
active research in the modal logic world on related topics.

For each fragment we have also proven a correspondence result, showing
the equivalence for a concrete functor to preserve the interpretation
of that fragment and for it to preserve certain categorical
structures. Such results have established a close connection between
the syntax and semantics of modal logic, and have deepened our
understanding of its abstract mathematical structures. They can be
seen as justification that topological category is a particularly nice
framework to explore its further connections with modal logic.

\subsection*{Acknowledgement}
In the process of preparing this paper, we are in great debt to many
extremely useful suggestion and constructive comments provided by
Johan van Benthem and Levin Hornischer. We would also like to thank
the annonymous referees for the helpful advice on the content and the
presentation of this paper.

\bibliographystyle{eptcs}
\bibliography{mybib}

\begin{thebibliography}{10}
\providecommand{\bibitemdeclare}[2]{}
\providecommand{\surnamestart}{}
\providecommand{\surnameend}{}
\providecommand{\urlprefix}{Available at }
\providecommand{\url}[1]{\texttt{#1}}
\providecommand{\href}[2]{\texttt{#2}}
\providecommand{\urlalt}[2]{\href{#1}{#2}}
\providecommand{\doi}[1]{doi:\urlalt{https://doi.org/#1}{#1}}
\providecommand{\eprint}[1]{arXiv:\urlalt{https://arxiv.org/abs/#1}{#1}}
\providecommand{\bibinfo}[2]{#2}

\bibitemdeclare{book}{adamek1990abstract}
\bibitem{adamek1990abstract}
\bibinfo{author}{Ji\u{r}\'{i} \surnamestart Ad\'{a}mek\surnameend},
  \bibinfo{author}{Horst \surnamestart Herrlich\surnameend} \&
  \bibinfo{author}{George~E. \surnamestart Strecker\surnameend}
  (\bibinfo{year}{1990}): \emph{\bibinfo{title}{Abstract and concrete
  categories : the joy of cats}}.
\newblock \bibinfo{publisher}{Wiley}, \bibinfo{address}{New York}.

\bibitemdeclare{article}{baltag2021lcd}
\bibitem{baltag2021lcd}
\bibinfo{author}{Alexandru \surnamestart Baltag\surnameend} \&
  \bibinfo{author}{Johan \surnamestart van Benthem\surnameend}
  (\bibinfo{year}{2021}): \emph{\bibinfo{title}{Knowability and Continuity: a
  topological account of epistemic dependence}}.
\newblock \bibinfo{note}{Unpublished draft}.

\bibitemdeclare{article}{baltag2021simple}
\bibitem{baltag2021simple}
\bibinfo{author}{Alexandru \surnamestart Baltag\surnameend} \&
  \bibinfo{author}{Johan \surnamestart van Benthem\surnameend}
  (\bibinfo{year}{2021}): \emph{\bibinfo{title}{A Simple Logic of Functional
  Dependence}}.
\newblock {\slshape \bibinfo{journal}{Journal of Philosophical Logic}}
  \bibinfo{volume}{50}(\bibinfo{number}{5}), pp. \bibinfo{pages}{939--1005},
  \doi{10.1007/s10992-020-09588-z}.

\bibitemdeclare{article}{topoapproach2019}
\bibitem{topoapproach2019}
\bibinfo{author}{Alexandru \surnamestart Baltag\surnameend},
  \bibinfo{author}{Nick \surnamestart Bezhanishvili\surnameend},
  \bibinfo{author}{Aybüke \surnamestart \"Ozg\"un\surnameend} \&
  \bibinfo{author}{Sonja \surnamestart Smets\surnameend}
  (\bibinfo{year}{2019}): \emph{\bibinfo{title}{A Topological Approach to Full
  Belief}}.
\newblock {\slshape \bibinfo{journal}{Journal of Philosophical Logic}}
  \bibinfo{volume}{48}(\bibinfo{number}{2}), pp. \bibinfo{pages}{205--244},
  \doi{10.1007/s10992-018-9463-4}.

\bibitemdeclare{article}{van2011dynamic}
\bibitem{van2011dynamic}
\bibinfo{author}{J.~\surnamestart van Benthem\surnameend} \&
  \bibinfo{author}{E.~\surnamestart Pacuit\surnameend} (\bibinfo{year}{2011}):
  \emph{\bibinfo{title}{Dynamic Logics of Evidence-Based Beliefs}}.
\newblock {\slshape \bibinfo{journal}{Studia Logica}}
  \bibinfo{volume}{99}(\bibinfo{number}{1}), p.~\bibinfo{pages}{61},
  \doi{10.1007/s11225-011-9347-x}.

\bibitemdeclare{article}{van2007dynamic}
\bibitem{van2007dynamic}
\bibinfo{author}{Johan \surnamestart van Benthem\surnameend}
  (\bibinfo{year}{2007}): \emph{\bibinfo{title}{Dynamic logic for belief
  revision}}.
\newblock {\slshape \bibinfo{journal}{Journal of Applied Non-Classical Logics}}
  \bibinfo{volume}{17}(\bibinfo{number}{2}), pp. \bibinfo{pages}{129--155},
  \doi{10.3166/jancl.17.129-155}.

\bibitemdeclare{book}{vanBenthem2010open}
\bibitem{vanBenthem2010open}
\bibinfo{author}{Johan \surnamestart van Benthem\surnameend}
  (\bibinfo{year}{2010}): \emph{\bibinfo{title}{{M}odal {L}ogic for {O}pen
  {M}inds}}.
\newblock \bibinfo{publisher}{Stanford: Center for the Study of Language and
  Information}.

\bibitemdeclare{book}{vanbenthem2011ldii}
\bibitem{vanbenthem2011ldii}
\bibinfo{author}{Johan \surnamestart van Benthem\surnameend}
  (\bibinfo{year}{2011}): \emph{\bibinfo{title}{Logical dynamics of information
  and interaction}}.
\newblock \bibinfo{publisher}{Cambridge University Press},
  \doi{10.1017/CBO9780511974533}.

\bibitemdeclare{incollection}{vanbenthem2012evidence}
\bibitem{vanbenthem2012evidence}
\bibinfo{author}{Johan \surnamestart van Benthem\surnameend},
  \bibinfo{author}{David \surnamestart Fern\'{a}ndez{-}Duque\surnameend} \&
  \bibinfo{author}{Eric \surnamestart Pacuit\surnameend}
  (\bibinfo{year}{2012}): \emph{\bibinfo{title}{Ecidence Logic: A New Look at
  Neighborhood Structures}}.
\newblock In \bibinfo{editor}{Marcus \surnamestart Kracht\surnameend},
  \bibinfo{editor}{Maarten \surnamestart de~Rijke\surnameend},
  \bibinfo{editor}{Heinrich \surnamestart Wansing\surnameend} \&
  \bibinfo{editor}{Michael \surnamestart Zakharyaschev\surnameend}, editors:
  {\slshape \bibinfo{booktitle}{Advances in Modal Logic}},
  \bibinfo{publisher}{CSLI Publications}, pp. \bibinfo{pages}{97--118}.

\bibitemdeclare{book}{blackburn2006handbook}
\bibitem{blackburn2006handbook}
\bibinfo{author}{Patrick \surnamestart Blackburn\surnameend},
  \bibinfo{author}{Johan~FAK \surnamestart van Benthem\surnameend} \&
  \bibinfo{author}{Frank \surnamestart Wolter\surnameend}
  (\bibinfo{year}{2006}): \emph{\bibinfo{title}{Handbook of modal logic}}.
\newblock \bibinfo{publisher}{Elsevier}.

\bibitemdeclare{book}{blackburnrijkevenema2001}
\bibitem{blackburnrijkevenema2001}
\bibinfo{author}{Patrick \surnamestart Blackburn\surnameend},
  \bibinfo{author}{Maarten \surnamestart de~Rijke\surnameend} \&
  \bibinfo{author}{Yde \surnamestart Venema\surnameend} (\bibinfo{year}{2001}):
  \emph{\bibinfo{title}{Modal Logic}}.
\newblock \bibinfo{series}{Cambridge Tracts in Theoretical Computer Science},
  \bibinfo{publisher}{Cambridge University Press},
  \doi{10.1017/CBO9781107050884}.

\bibitemdeclare{book}{cina2017thesis}
\bibitem{cina2017thesis}
\bibinfo{author}{G.~\surnamestart Cin\`{a}\surnameend} (\bibinfo{year}{2017}):
  \emph{\bibinfo{title}{Categories for the working modal logician}}.
\newblock \bibinfo{publisher}{ILLC Dissertation}.

\bibitemdeclare{book}{goble2006deontic}
\bibitem{goble2006deontic}
\bibinfo{author}{Lou \surnamestart Goble\surnameend} \&
  \bibinfo{author}{John-Jules~Ch \surnamestart Meyer\surnameend}
  (\bibinfo{year}{2006}): \emph{\bibinfo{title}{Deontic Logic and Artificial
  Normative Systems: 8th International Workshop on Deontic Logic in Computer
  Science, DEON 2006, Utrecht, The Netherlands, July 12-14, 2006,
  Proceedings}}.
\newblock \bibinfo{volume}{4048}, \bibinfo{publisher}{Springer Berlin,
  Heidelberg}, \doi{10.1007/11786849}.

\bibitemdeclare{book}{hintikka1962knowledge}
\bibitem{hintikka1962knowledge}
\bibinfo{author}{Jaakko \surnamestart Hintikka\surnameend}
  (\bibinfo{year}{1962}): \emph{\bibinfo{title}{Knowledge and Belief: An
  Introduction to the Logic of the Two Notions}}.
\newblock \bibinfo{publisher}{Ithaca: Cornell University Press}.

\bibitemdeclare{book}{hofmann2014monoidal}
\bibitem{hofmann2014monoidal}
\bibinfo{author}{Dirk \surnamestart Hofmann\surnameend},
  \bibinfo{author}{Gavin~J \surnamestart Seal\surnameend} \&
  \bibinfo{author}{Walter \surnamestart Tholen\surnameend}
  (\bibinfo{year}{2014}): \emph{\bibinfo{title}{Monoidal Topology: A
  Categorical Approach to Order, Metric, and Topology}}.
\newblock \bibinfo{volume}{153}, \bibinfo{publisher}{Cambridge University
  Press}, \doi{10.1017/CBO9781107517288}.

\bibitemdeclare{book}{joyal1984extension}
\bibitem{joyal1984extension}
\bibinfo{author}{Andr{\'e} \surnamestart Joyal\surnameend} \&
  \bibinfo{author}{Myles \surnamestart Tierney\surnameend}
  (\bibinfo{year}{1984}): \emph{\bibinfo{title}{An extension of the Galois
  theory of Grothendieck}}.
\newblock \bibinfo{volume}{309}, \bibinfo{publisher}{American Mathematical
  Soc.}

\bibitemdeclare{article}{kishida2017categories}
\bibitem{kishida2017categories}
\bibinfo{author}{Kohei \surnamestart Kishida\surnameend}
  (\bibinfo{year}{2017}): \emph{\bibinfo{title}{Categories for Dynamic
  Epistemic Logic}}.
\newblock {\slshape \bibinfo{journal}{Electronic Proceedings in Theoretical
  Computer Science}} \bibinfo{volume}{251}, pp. \bibinfo{pages}{353--372},
  \doi{10.4204/EPTCS.251.26}.

\bibitemdeclare{book}{liu2011reasoning}
\bibitem{liu2011reasoning}
\bibinfo{author}{Fenrong \surnamestart Liu\surnameend} (\bibinfo{year}{2011}):
  \emph{\bibinfo{title}{Reasoning about preference dynamics}}.
\newblock \bibinfo{volume}{354}, \bibinfo{publisher}{Springer Science \&
  Business Media}, \doi{10.1007/978-94-007-1344-4}.

\bibitemdeclare{article}{mckinsey1944algebra}
\bibitem{mckinsey1944algebra}
\bibinfo{author}{John Charles~Chenoweth \surnamestart McKinsey\surnameend} \&
  \bibinfo{author}{Alfred \surnamestart Tarski\surnameend}
  (\bibinfo{year}{1944}): \emph{\bibinfo{title}{The algebra of topology}}.
\newblock {\slshape \bibinfo{journal}{Annals of mathematics}}, pp.
  \bibinfo{pages}{141--191}, \doi{10.2307/1969080}.

\bibitemdeclare{inproceedings}{plaza1989logics}
\bibitem{plaza1989logics}
\bibinfo{author}{Jan \surnamestart Plaza\surnameend} (\bibinfo{year}{1989}):
  \emph{\bibinfo{title}{Logics of public announcements}}.
\newblock In: {\slshape \bibinfo{booktitle}{Proceedings 4th International
  Symposium on Methodologies for Intelligent Systems}}, pp.
  \bibinfo{pages}{201--216}.

\bibitemdeclare{article}{plaza2007pal}
\bibitem{plaza2007pal}
\bibinfo{author}{Jan \surnamestart Plaza\surnameend} (\bibinfo{year}{2007}):
  \emph{\bibinfo{title}{Logics of public communications}}.
\newblock {\slshape \bibinfo{journal}{Synthese}}
  \bibinfo{volume}{158}(\bibinfo{number}{2}), pp. \bibinfo{pages}{165--179},
  \doi{10.1007/s11229-007-9168-7}.

\bibitemdeclare{incollection}{scott1970advice}
\bibitem{scott1970advice}
\bibinfo{author}{Dana \surnamestart Scott\surnameend} (\bibinfo{year}{1970}):
  \emph{\bibinfo{title}{Advice on modal logic}}.
\newblock In: {\slshape \bibinfo{booktitle}{Philosophical problems in logic}},
  \bibinfo{publisher}{Springer}, pp. \bibinfo{pages}{143--173},
  \doi{10.1007/978-94-010-3272-8_7}.

\bibitemdeclare{article}{tamminga2021expressivity}
\bibitem{tamminga2021expressivity}
\bibinfo{author}{Allard \surnamestart Tamminga\surnameend},
  \bibinfo{author}{Hein \surnamestart Duijf\surnameend} \&
  \bibinfo{author}{Frederik \surnamestart Van De~Putte\surnameend}
  (\bibinfo{year}{2021}): \emph{\bibinfo{title}{Expressivity results for
  deontic logics of collective agency}}.
\newblock {\slshape \bibinfo{journal}{Synthese}}
  \bibinfo{volume}{198}(\bibinfo{number}{9}), pp. \bibinfo{pages}{8733--8753},
  \doi{10.1007/s11229-020-02597-0}.

\bibitemdeclare{article}{vonwright1954essay}
\bibitem{vonwright1954essay}
\bibinfo{author}{Georg~H. \surnamestart von Wright\surnameend}
  (\bibinfo{year}{1954}): \emph{\bibinfo{title}{An essay in modal logic}}.
\newblock {\slshape \bibinfo{journal}{British Journal for the Philosophy of
  Science}} \bibinfo{volume}{5}(\bibinfo{number}{18}).

\end{thebibliography}

\appendix
\section{Proof of Theorem~\ref{thm:preserveproducttypeupdate}}
\label{app}
To complete the other half of the proof, we roughly need to show that
any initial lift and any finite meets in the fibre could be
represented by some product type update with a specific chosen product
type. First of all, since empty product type update is included in our
dynamic extension $\mc L\PRO$, to preserve it we may assume $F$
already preserves the top element within each fibre. We first show
that $F$ commutes with pullback maps. Suppose for some function
$\pi : E \to X$, $F$ does not commute with the initial lifts on $\pi$
in $\mc A$ and $\mc B$. This means that we have some object $A$ in the
fibre $\mc A_X$ , such that $F\pi^*A$ and $\pi^*FA$ are two distinct
objects in $\mc B_X$. Now since the semantic functor on $\mc B$ is
injective on objects, the induced operators $(\pi^*FA)\upa_{\mc B}$,
which we denote as $m$, and $(F\pi^*A)\upa_{\mc B}$, which we denote
as $m'$, will be distinct, which means they disagree on some subset
$T$ of $E$.

Now consider the product type
$\ms E = \pair{E,\top^{\mc A}_E,W,\set{p_e}_{e\in E}}$, where the
family of formulas is an $E$-indexed family of distinct propositional
letters. For the evaluation function $W$ on $E$, we require that for
some propositional letter $q$ distinct from $p_e$ for any $e\in E$, we
have $W(q) = T$. Now consider an evaluation function $V$ on $X$, such
that for any $e \in E$ we have $V(p_e) = \set{\pi(e)}$, which means
that $\assv{p_e}_A$ is a singleton for any $e \in E$. We also requires
that $V(q) = X$. Then by definition, we have
\[ E \otimes_V X = \sum_{e\in E}\assv{p_e}_A = E, \] and it is not
hard to see that the projection map $\pi_E$ is the identity on $E$,
and $\pi_X$ is simply given by $\pi$. Notice that, the above statement
of the underlying set of the updated model remains true even if we
have calculated it in $\mc B$.

Now by definition, the geometric data on the updated model is
calculated as follows,
\[ \ms E \otimes_V A = 1_E^*\top^{\mc A}_E \wedge \pi^*A = \pi^*A, \]
and for the induced product update in $\mc B$,
\[ F\ms E \otimes_V FA = 1^*_EF\top^{\mc A}_E \wedge \pi^*FA =
  \top^{\mc B}_E \wedge \pi^*FA = \pi^*FA. \]
The above uses the fact that initial lifts preserves top elements
since it is a right adjoint, and the assumption that $F$ preserves top
elements in the fibres. As for the evaluation function $W \otimes V$,
it is easy to calculate that
\[ (W \otimes V)(q) = W(q) \wedge \pi\inv V(q) = W(q) = T. \] Finally,
consider the interpretation of the formula $\pair{\ms E,\set e}\Box
q$, where $e$ is some element in $E$ such that $e \in m(T)$ but $e
\not\in m'(T)$ (or the other way around). Then by definition, we have
the following calculation,
\[ \assv{\pair{\ms E,\set e}\Box q}_A = \exists_{\pi}(\set{(e,\pi(e))}
\cap \ass{\Box q}^{W \otimes V}_{\pi^*A}) =
\exists_{\pi}(\set{(e,\pi(e))} \cap \ass{\Box q}^{W \otimes
V}_{F\pi^*A}) = \emptyset. \] The first equality is due to the fact
that $E \otimes_V A = \pi^*A$ as we have shown above; the second
equality is by the fact that $F$ preserves the interpretation of $\mc
L$; and the final equality holds because we have assumed $e \not\in
m'(T)$. On the other hand, we have the other calculation as follows,
\[ \assv{\pair{F\ms E,e}\Box q}_{FA} = \exists_{\pi}(\set{(e,\pi(e))}
\cap \ass{\Box q}^{W \otimes V}_{F\ms E \otimes_V FA}) =
\exists_{\pi}(\set{(e,\pi(e))} \cap \ass{\Box q}^{W \otimes
V}_{\pi^*FA}) = \set{\pi(e)}. \] These calculation are basically the
same as before, only that in the final step, the result is a singleton
$\set{\pi(e)}$ because $e \in m(T)$. This constructions shows that $F$
would then not preserve the interpretation of the formula $\pair{\ms
E,\set e}\Box q$ on this particular model. Hence, $F$ must preserves
the initial lift of any single structured sources.

Furthermore, we need to show that $F$ preserves the binary meets
fibre-wise as well. The basic idea is the same. Suppose $F$ does not
preserve binary meets in the fibre, then for some set $X$ and some
$A$, $B$ in the fibre $\mc A_X$, we would have $F(A \wedge B)$
distinct from $FA \wedge FB$ in $\mc B_X$. Again, the operators $m,m'$
associated to $F(A \wedge B)$ and $FA \wedge FB$ would differ on some
subset $T$ of $X$; we let $y\in X$ be the element in $m(T)$ but not
$m'(T)$ (or the other way around). We can then construct the product
type $\ms X$ as follows,
\[ \ms X = \pair{X,B,\top^{\Evl}_X,\set{p_x}_{x\in X}}. \] We also
consider the model $A$ on $X$, with a chosen evaluation function $V$
satisfying the following condition: For any $x \in X$, we have $V(p_x)
= \set x$, and for another distinct variable $q$ we have $V(q) =
T$. The product type update would result in the following model,
\[ X \otimes_V X = \sum_{x\in X}\assv{p_x}_A = X, \] and the two
projection maps are both the identity function $1_X$ on $X$. Again,
this is independent of the modal categories $\mc A$ or $\mc B$. The
topology categorical structure on the updated model, calculated in
$\mc A$, is simply given as follows,
\[ \ms X \otimes_V A = 1_X^*B \wedge 1_X^*A = A \wedge B. \] In the
modal category $\mc B$ however, we have
\[ F\ms X \otimes_V FA = 1^*XFB \wedge 1^*XFA = FA \wedge FB. \] In
both cases, it is easy to see that the updated evaluation function
$\top^{\Evl}_X \otimes V$ remains to be $V$ itself.

By definition, consider the evaluation of the formula $\pair{\ms
X,\set y}\Box q$. On one hand,
\[ \assv{\pair{\ms X,\set y}\Box q}_A = \set y \cap \assv{\Box q}_{A
\wedge B} = \set y \cap \assv{\Box q}_{F(A \wedge B)} = \set y \cap
m(T) = \set{y}. \] On the other hand,
\[ \assv{\pair{F\ms X,\set y}\Box q}_{FA} = \set y \cap \assv{\Box
q}_{FA \wedge FB} = \set y \cap m'(T) = \emptyset. \] Hence, this
explicitly constructs a formula where $F$ does not preserve its
interpretation, and this completes the proof.

\end{document}